\documentclass[12pt,a4]{amsart}
\usepackage[a4paper, left=28mm, right=28mm, top=28mm, bottom=34mm]{geometry}
\usepackage[all]{xy}
\usepackage{amsmath}
\usepackage{amssymb}
\usepackage{amsthm}
\usepackage{url}
\usepackage{mathrsfs}

\theoremstyle{definition}
\newtheorem{thm}{Theorem}[section]
\newtheorem{lem}[thm]{Lemma}
\newtheorem{cor}[thm]{Corollary}
\newtheorem{prop}[thm]{Proposition}

\theoremstyle{definition}

\newtheorem{rem}[thm]{Remark}

\newtheorem{dfn}[thm]{Definition}

\newtheorem{ex}[thm]{Example}

\newcommand{\codim}{\mathrm{codim}}
\newcommand{\divisor}{\mathrm{div}}
\newcommand{\Ker}{\mathrm{Ker}}
\newcommand{\sing}{\mathrm{sing}}
\newcommand{\red}{\mathrm{red}}
\newcommand{\Nis}{\mathrm{Nis}}
\newcommand{\zar}{\mathrm{zar}}
\newcommand{\cdh}{\mathrm{cdh}}

\newcommand{\Spec}{\mathrm{Spec}}

\newcommand{\CH}{\mathrm{CH}}
\newcommand{\hocolim}{\mathrm{hocolim}}
\newcommand{\reg}{\mathrm{reg}}

\numberwithin{equation}{section}
\begin{document}

\title[The 0-th Suslin homology and Chow groups of 0-cycles with modulus]{the comparison of the 0-th Suslin homology and Chow groups of 0-cycles with modulus}

\author{Teppei Nakamura}
\address{Department of Mathematics, Faculty of Science, Kyoto University, Kyoto 606-8502, Japan}
\email{nakamura.teppei.27e@st.kyoto-u.ac.jp}

\date{\today}

\begin{abstract}
We show that, for a $K_0$-regular projective normal surface $X$ over a perfect field $k$ of positive characteristic and a reduced effective Cartier divisor $D\hookrightarrow X$, the Chow group of zero cycles on $X$ with modulus $D$ coincides with the 0-th Suslin homology of $X\setminus D$. Moreover, we show that this isomorphism also holds for a projective smooth scheme of any dimension with a reduced effective Cartier divisor.
\end{abstract}

\maketitle

\section{Introduction}
\subsection{Main results}
The goal of geometric class field theory is to describe the abelian fundamental group of a smooth scheme $X$ over a finite field $k$ by some invariants of $X$ itself. For an effective divisor $D$, the Chow group of 0-cycles with modulus $\CH_0(X|D)$ and the 0-th Suslin homology $H^S_0(X\setminus D)$ play important roles. More precisely, for a proper normal scheme $X$ and an effective Cartier divisor $D\hookrightarrow X$ on $X$ such that its complement $X\setminus D$ is smooth, Binda, Kerz, Krishna, and Saito (\cite{BKS}, \cite{Kerz-Saito}) showed the isomorphism
\begin{align*}
    \CH_0(X|D)_0\stackrel{\cong}{\longrightarrow} \pi^{ab}_1(X|D)_0,
\end{align*}
where each of $\CH_0(X|D)_0\subset \CH_0(X|D)$ and $\pi^{ab}_1(X|D)_0\subset \pi^{ab}_1(X|D)$ is the degree zero subgroup. On the side of the tame class field theory, for a smooth scheme $U$, Schmidt \cite{Schmidt}, \cite{SS} showed the isomorphism 
\begin{align*}
     H^S_0(U)_0\stackrel{\cong}{\longrightarrow} \pi^{t,ab}_1(U)_0,
\end{align*}
where the subscript $0$ means the same as above. Obviously, there is a commutative diagram of finite abelian groups
\begin{align*}
    \xymatrix{
\CH_0(X|D)_0 \ar[r]^{\cong} \ar@{->>}[d] & 
\pi^{ab}_1(X|D)_0 \ar@{->>}[d]\\
H^S_0(X\setminus D)_0 \ar[r]^{\cong} & \pi^{t,ab}_1(X\setminus D)_0
    }
\end{align*}
Naturally, when $D$ is reduced, there arises a question whether these two abelian fundamental groups coincide, or equivalently, the Chow group of 0-cycles with modulus and the 0-th Suslin homology coincide. Of course, if $X$ is a projective smooth curve, the isomorphism $\CH_0(X|D)_0 \cong H^S_0(X\setminus D)_0$ is obvious. However, in the higher dimensional case, it is not known that they coincide.

The main purpose of this paper is to prove the following theorem. 
\begin{thm}\label{main theorem}
  Let $k$ be a perfect field of positive characteristic. Let $X$ be a projective normal scheme over $k$, and let $D\hookrightarrow X$ be a reduced effective Cartier divisor on $X$. Assume that one of the following conditions holds:
 \begin{enumerate}
     \item $X$ is a $K_0$-regular surface over $k$.
     \item $X$ is a projective smooth scheme over $k$.
 \end{enumerate}
  Then the canonical map
  \begin{align*}
    \CH_{0}(X|D)\to H^{S}_{0}(X\setminus D)
  \end{align*}
 is an isomorphism.
\end{thm}

 In the previous paper, Binda and Krishna  \cite[Theorem 1.1]{BK} showed the isomorphism of this canonical map under one of the following conditions:

  \begin{enumerate}
    \item $X$ is a projective smooth scheme over a field $k$, and $D$ is a simple normal crossing divisor on $X$.
    \item $k$ is a perfect field, $X$ is a projective smooth scheme over $k$, $D$ is a seminormal divisor on $X$, and $\dim(X)\leq 2$.
    \item $X$ is projective smooth scheme over an algebraically closed field of positive characteristic.
  \end{enumerate}
  
Hence, for a smooth projective scheme over a perfect field $k$ of positive characteristic, we can get rid of some hypotheses of $D$ in \cite[Theorem 1.1]{BK}. Moreover we can show that the canonical isomorphism also holds for $K_0$-regular normal surfaces.  Since a regular scheme is also a $K_n$-regular scheme for any $n\in \mathbb{Z}$, our results are the generalization of \cite[Theorem 1.1]{BK} for surfaces, in which $X$ is a projective smooth scheme over $k$.

\subsection{Overview of proofs}
First, we will show the main theorem for a surface $X$ and a reduced Cartier divisor $D$. The strategy of the proof is similar to that of Binda and Krishna, which reduces the computations of cycle groups into those of algebraic $K$-groups and homotopy $K$-groups. 

According to Bloch's formula, it has been known that the Chow group of 0-cycles with modulus is isomorphic to the subgroup of relative $K$-group $K_0(X,D)$. So we need to find out the connection between $K_0(X,D)$ and $KH_0(X,D)$. 

One of the critical points in this paper is Corollary \ref{SK-SKH iso}, which shows the injectivity $K_1(D) \hookrightarrow KH_1(D)$  for reduced curves $D$ over a perfect field of positive characteristic. A similar isomorphism was proven in \cite{BK} over a field of zero characteristic. The author proves Corollary \ref{SK-SKH iso} by using Dennis-Stein's result that $ K_2(E) \cong K_2(E_{\red})$ is an isomorphism for a dimension zero variety $E$ over a perfect field of positive characteristic.  This result is used to show the injectivity $K_0(X,D) \hookrightarrow KH_0(X,D)$. Notice that Corollary \ref{SK-SKH iso} is false for a perfect field of zero characteristic. Consequently, there exists a counterexample to our main results for varieties over a field of zero characteristic (see, for example, \cite[Theorem 4.4]{Binda-Krishna}). 

On the other hand, the computation of the 0-th Suslin homology  $H_0(X \setminus D)$ is slightly different. Once we take a smooth surface $\overline{X}$ and a simple normal crossing divisor $(\pi^*D)_{\red}$ for a resolution of singularities $\pi :\overline{X}\to X$ such that $\overline{X}\setminus \pi^*D\cong X\setminus D$, $H_0(X \setminus D)$ is isomorphic to $\CH_0(\overline{X}|(\pi^*D)_{\red})$ by the result of Binda and Krishna. Hence $H_0(X \setminus D)$ is isomorphic to the subgroup of $KH_0(\overline{X},(\pi^*D)_{\red})$.

Here, another important point in my proof is Theorem \ref{Cisinski} proved by Cisinski(\cite{Cisinski}), which implies $KH(X,D)\cong KH(\overline{X}, (\pi^*D)_{\red})$. This shows the isomorphism of $\CH_0(X|D)$ and $H_0(X \setminus D)$ in the surface case. Notice that the same argument as in \cite{BK} by Binda and Krishna does not work for singular surfaces. More precisely, if $X$ is singular, the position of the curve in $X$ which produces a rational equivalence of $H_0(X \setminus D)$ or $\CH_0(X|D)$ might be too bad to explore those rational equivalences in $X$.  Thus, it is necessary to consider those rational equivalences in $\overline{X}$ rather than in $X$ itself. Actually, the invariants of the pairs $(\overline{X},\pi^*D)$ and $(X,D)$ can not be the same. However, Cisinski's theorem guarantees the success of my proof.

In the higher dimensional case, we use Bertini theorems in order to reduce the higher dimensional case to the surface case. We prove it by an inductive argument, which is often used in the higher class field theory.

In \S~\ref{sec:cycles}, we recall the definition of Chow groups of 0-cycles with modulus and the 0-th Suslin homology. In \S~\ref{sec:review of K-theory}, we collect the results of algebraic $K$-theory and homotopy $K$-theory we will need. In \S~\ref{sec:calculation of K-theory}, we compute various algebraic and homotopy $K$-groups. In \S~\ref{sec:proof of 1}, we prove the main theorem (1). After the review of Bertini theorems in \S~\ref{sec:Bertini}, we prove the main theorem (2) in \S~\ref{sec:proof of 2}.

\subsection*{Acknowledgements}
The author would like to thank his advisor, Tetsushi Ito, for useful discussions and warm encouragement. He gave me an invaluable suggestions in the writing of this paper. This work was supported by JST SPRING, Grant Number JPMJSP2110.

\section{Review of zero cycles}\label{sec:cycles}

  Let $X$ be an integral scheme over $k$ which is regular in codimension 1, $D$ an effective Cartier divisor of $X$, and $A\hookrightarrow D$ a reduced closed subscheme of $\codim_X(A)\geq 2$. Assume that $X_{\sing} \subset D$. Let $Z_{0}(X \setminus D)$ be the abelian group formally generated by closed points of $X \setminus D$. We shall say that a finite morphism $p:C\to X$ is \emph{admissible for $A$} if $C$ is the integral normal curve, $p(C)\not\subset D$ , and  $p(C)\cap A=\varnothing$. If $A=\varnothing$, we simply call it admissible. Let $R_{0}(X,D;A)\subset Z_{0}(X\setminus D)$ be the subgroup generated by $p_*(\divisor(f))$, where $p:C\to X$ is $p$ is admissible for $A$ and $f\in \Ker(\mathcal{O}^{\times}_{C,p^{*}(D)}\to \mathcal{O}^{\times}_{p^{\ast}D})$.

\begin{dfn}
   $\CH_{0}(X|D;A)$ is defined to be $Z_{0}(X\setminus D)$/$R_{0}(X,D;A)$. If $A=\varnothing$, $\CH_{0}(X|D)$ is defined to be $\CH_{0}(X|D;\varnothing)$, which is generally called the \emph{Chow group of 0-cycles on $X$ with modulus $D$}.
\end{dfn}

  Let $R_{0}^S(X,D)\subset Z_{0}(X\setminus D)$ be the subgroup generated by $p_*(\divisor(f))$ , where $p:C\to X$ is admissible and  $f\in \Ker(\mathcal{O}^{\times}_{C,p^{*}D}\to \mathcal{O}^{\times}_{(p^{*}D)_{\red}})$. Here we shall define $0$-th Suslin homology. Of course this should be defined by $0$-th homology of a complex introduced by Suslin and Voevodsky \cite{SV}. But we have another convenient definition as follows, which was proved by Schmidt \cite[Theorem 5.1]{Schmidt}.
\begin{dfn}
 The \emph{0-th Suslin homology} $H_0^S(X\setminus D)$ is defined to be $Z_{0}(X,D)$/$R_{0}^S(X,D)$.
\end{dfn}

Let $p:X'\to X$ be a morphism between proper normal $k$-scheme, $D\hookrightarrow X$ an effective Cartier divisor on $X$ such that $D':=X'\times_X D$ is an effective Cartier divisor and ${X'}_{\sing}\subset D'$. Then, by the previous definitions, we can see immediately that there exist maps
\[
\CH_0(X'|D')\to \CH_0(X|D)\qquad \CH_0(X|D)\twoheadrightarrow H^S_0(X\setminus D)
\]
\[
\CH_0(X|mD)\twoheadrightarrow \CH_0(X|nD)\quad (0\leq n\leq m).
\]

\section{Review of algebraic K-theory and homotopy K-theory}\label{sec:review of K-theory}
\subsection{Recollection of algebraic $K$-theory}
Let $X$ be a quasi-projective scheme over a field $k$. Let $K(X)$ denote the non-connective $K$-theory spectrum of $X$ and $\mathcal{K}_{n,X}$ be the sheafification of $K_n$-group on $X_{\zar}$. In particular, there exists a spectral sequence due to Thomason-Trobaugh (see \cite[Theorem 10.3]{TT})
 \begin{align*}
      \xymatrix{
  E^{p,q}_{2,\zar} = H^p_{\zar}(X,\mathcal{K}_{q,X})  & \Rightarrow & K_{q-p}(X). }
 \end{align*}
 \begin{rem}
  In the above spectral sequence, we may replace the Zariski topology with Nisnevich topology. For details, see \cite[Theorem 10.8]{TT}.
\end{rem}
By the above spectral sequence, we have a homomorphism $K_0(X)\to H^0_{\zar}(X,\mathbb{Z}) $. We define $\widetilde{K_0(X)}$ to be the kernel of its morphism. Again, by the above spectral sequence, we have a canonical homomorphism $\widetilde{K_0(X)}\to H^1_{\zar}(X,\mathcal{O}^{\times}_{X})$. We let $SK_0(X)$ be the kernel of the homomorphism. Krishna \cite[Lemma 2.1]{Krishna} proved that the edge map $K_1(X)\to H^0_{\zar}(X,\mathcal{O}^{\times}_{X})$ induced by the above spectral sequence is surjective under the assumption $\dim(X)\leq 2$. Moreover this map splits functorially in $X$. Note that the similar argument holds for Nisnevich topoology and we get the isomorphism $H^1_{\zar}(X,\mathcal{O}^{\times}_{X})\cong H^1_{\Nis}(X,\mathcal{O}^{\times}_{X})$ by Hilbert's theorem 90. Hence if $\dim(X)\leq 2$, we have an isomorphism and an exact sequence 
\[
H^2_{\zar}(X,\mathcal{K}_{2,X})\cong  H^2_{\Nis}(X,\mathcal{K}_{2,X})\cong SK_0(X),
\]
\[
   \xymatrix{
   0 \ar[r] & SK_0(X) \ar[r] & \widetilde{K_0(X)} \ar[r] & H^1_{\zar}(X,\mathcal{O}^{\times}_{X}) \ar[r] & 0.} 
\]
 We shall define $SK_1(X)$ to be the kernel of the homomorphism $K_1(X)\twoheadrightarrow H^0_{\zar}(X,\mathcal{O}^{\times}_{X})$. Therefore we have an exact sequence 
\begin{align*}
     \xymatrix{
   0 \ar[r] & SK_1(X) \ar[r] & K_1(X) \ar[r] & H^0_{\zar}(X,\mathcal{O}^{\times}_{X}) \ar[r] & 0.} 
\end{align*}
For a closed subscheme $D\hookrightarrow X$, we define the relative K-theory $K(X,D)$ to fit into the homotopy fiber sequence
\begin{align*}
    K(X,D) \to K(X) \to K(D).
\end{align*}
Since the relative $K$-theory $K(X,D)$ also satisfies the Mayer-Vietoris property, the same argument goes through. Then we shall define $SK_0(X,D)$ and $SK_1(X,D)$ in the same manner. We can apply the same argument as above to relative case. Therefore, if $\dim(X)\leq 2$ we also have isomorphisms
\[
H^2_{\zar}(X,\mathcal{K}_{2,(X,D)})\cong  H^2_{\Nis}(X,\mathcal{K}_{2,(X,D)})\cong SK_0(X,D).
\]

\subsection{Recollection of homotopy $K$-theory}
 Put 
\[
 \Delta_{n}:=\Spec(\mathbb{Z}[t_{0},\cdots,t_{n}]/(\sum_{i}t_{i}-1)).
\]

 Weibel \cite{Weibel-KH} defined the homotopy $K$-theory spectrum and relative homotopy K-theory spectrum as follows:
 \[
 KH(X):=\underset{n} {\hocolim}(K(X\times \Delta_{n}))\]
 \[KH(X,D):=\underset{n} {\hocolim}(K(X\times \Delta_{n}), D\times \Delta_{n}).
 \]
Since the homotopy colimit preserves homotopy fiber sequences (see \cite[Lemma 5.19]{Thomason}), there exists a homotopy fiber sequence
\begin{align*}
    KH(X,D)\to KH(X)\to KH(D),
\end{align*}
where $D\hookrightarrow X$ is a closed subscheme of $X$. Since $KH(X)$ is homotopy equivalent to the cdh-fibrant replacement of $K(X)$(see \cite[Th\'eor\`eme 3.9]{Cisinski}), there exists a commutative diagram of strongly convergent spectral sequences 
\begin{align*}
    \xymatrix{
  E^{p,q}_{2,\zar} = H^p_{\zar}(X,\mathcal{K}_{q,X}) \ar[d] &  \Rightarrow & K_{q-p}(X) \ar[d] \\
 E^{p,q}_{2,\cdh} = H^p_{\cdh}(X, \mathcal{K}_{q,X}) & \Rightarrow & KH_{q-p}(X).}
\end{align*}
Similarly, we shall define $SKH_0(X)$ and $SKH_1(X)$ as well as the algebraic $K$-theory. 
\begin{rem}
    Actually, $KH(X)\cong KH((X)_{\red})$ holds for any noetherian scheme $X$ since this isomorphism holds for any affine scheme (see \cite[I\hspace{-1.2pt}V Corollary 12.5]{K-book}) and homotopy $K$-theory spectrum satisfies the Mayer-Vietoris property.
\end{rem}

\subsection{Review of $K_n$-regularity}
Next, we recall the definition and some properties of $K_n$-regular schemes. Throughout this section, let $X$ be a noetherian scheme.
\begin{dfn}\label{1.2}
For $n\in \mathbb{Z}$, we say that a scheme $X$ is \emph{$K_n$-regular} if the canonical map
\begin{align*}  
    K_n(X) \to K_n(X\times \mathbb{A}^m)
\end{align*}
is an isomorphism for every $m\geq 0$.
\end{dfn}
\begin{ex} \qquad

\begin{enumerate}
    \item If $X$ is a regular scheme, $X$ is $K_n$-regular for all $n\in \mathbb{Z}$ (see \cite[V Corollary 6.13.1]{K-book}).
    \item Let $X$ be a normal surface over a perfect field $k$. If $X$ has rational singularities, $X$ is a $K_0$-regular scheme (see \cite[Example 2.13, Theorem 5.9, Theorem 8.4]{Weibel}).
\end{enumerate}
   
\end{ex}

\begin{lem}\label{Vorst}
   Let $X$ be a $K_n$-regular scheme for an integer $n\in \mathbb{Z}$. Then $X$ is also a $K_{n-1}$-regular scheme. 
\end{lem}
\begin{proof}
    See \cite[V Theorem 8.6]{K-book}.
\end{proof}

The next lemma describes the relation of the algebraic $K$-theory and the homotopy $K$-theory for a $K_n$-regular scheme.
\begin{lem}\label{K-reg}
Let $X$ be a $K_n$-regular scheme. Then we have an isomorphism
\[ K_n(X)\stackrel{\cong}{\longrightarrow} KH_n(X), \]
and a surjection
\[ K_{n+1}(X)  \twoheadrightarrow KH_{n+1}(X). \]
In particular, if $X$ is a regular scheme, we have an isomorphism $K_n(X)\stackrel{\cong}{\longrightarrow} KH_n(X)$ for every $n\in \mathbb{Z}$.
\end{lem}

\begin{proof}
  First, there is a right half plane spectral sequence 
  \begin{align*}
      \xymatrix{
  E_{p,q}^1 = K_q(X \times \Delta_p) \Rightarrow KH_{p+q}(X)
}
  \end{align*}
 Here, the map $d_1:E^1_{p,q}\to E^1_{p-1,q}$ is given by $d_1=\sum_{0\le i \le p}(-1)^{i}\partial^*_{i}$, where $\partial^*_{i}$ is induced by a closed immersion
 \[
 \partial_{i}:
 \Spec(k[t_{0},\cdots,t_{i-1},t_{i+1}\cdots, t_{p}]/(\sum_jt_j-1)])\hookrightarrow 
 \Spec(k[t_{0},\cdots, t_{p}]/(\sum_it_i-1)]).
 \]

By Lemma \ref{Vorst}, $X$ is a $K_q$-regular scheme for $q\leq n$. Hence we have $E^2_{0,q} = K_q(X)$ for all $q\leq n$. Moreover, we have $E^2_{p,q} = 0$ for all $p>0$ and $q\leq n$.

Furthermore, $E^2_{n+1,0}$ is a quotient of $K_{n+1}(X)$. This implies our claims.
\end{proof}

Before going to another crucial property of the homotopy $K$-theory, we recall the definition of an abstract blow-up square.
\begin{dfn}
    A scheme cartesian square 
    \begin{align*}
        \xymatrix{
        \overline{Z} \ar[r]^{i'} \ar[d]_{p'} & \overline{X} \ar[d]^{p} \\
        Z \ar[r]^{i} & X
        }
    \end{align*}
    is called an \emph{abstract blow-up square} if $p:\overline{X}\to X$ is proper, $i:Z \to X$ is a closed immersion, and the induced morphism $p:\overline{X}\setminus \overline{Z}\to X\setminus Z$ is an isomorphism.
\end{dfn}

\begin{thm}\label{Cisinski}
For any abstract blow-up square, 
\begin{align*}
     \xymatrix{
  KH(X) \ar[r] \ar[d] & {KH(Z)} \ar[d]  \\
  KH(\overline{X}) \ar[r] & {KH(\overline{Z})}}
\end{align*}
is a homotopy cartesian square of spectra.
\end{thm}
\begin{proof}
See \cite[Theorem 3.9]{Cisinski}.
\end{proof}

Finally, we recall the definition and a property of a seminormal scheme.
\begin{dfn}
Let $A$ be a reduced ring over a field $k$, and let $B$ be the integral closure in its total quotient. We shall say that a subring $A^{sn}\subset B$ is the \emph{seminormalization} of $A$ if $A^{sn}$ is the maximal between subrings $A'\subset B$ which contain $A$ and satisfy the following conditions: 
\begin{enumerate}
    \item For any $x\in \Spec(A)$, there is only one $x'\in \Spec(B)$ lying above $x$.
    \item For any $x\in \Spec(A)$, the canonical map $k(x)\to k(x')$ is an isomorphism, where $x'\in \Spec(B)$ is lying above $x$.
\end{enumerate}
We shall say that $A$ is \emph{seminormal} if $A$ is the seminormalization of itself. For any reduced noetherian scheme $X$ over a field $k$, we define its seminormalization $X^{sn}$ in a similar way.
\end{dfn}

\begin{lem}\label{seminormal}
    Let $X$ be a reduced separated scheme of finite type over a field $k$. Then there is a commutative diagram of abelian groups 
    \begin{align*}
\xymatrix{
 H^0_{\zar}(X,\mathcal{O}^{\times}_{X}) \ar[r] \ar[d] &
  H^0_{\zar}(X^{sn}, \mathcal{O}^{\times}_{X^{sn}})\ar[d]^{\cong} \\
   H^0_{\cdh}(X, \mathcal{O}^{\times}_X) \ar[r]^-{\cong} &
   H^0_{\cdh}(X^{sn}, \mathcal{O}^{\times}_{X^{sn}}).
}
    \end{align*}
 Moreover, the bottom horizontal map and the right vertical map are isomorphisms.

\end{lem}
\begin{proof}
See \cite[Lemma 2.1]{BK}.    
\end{proof}

\section{Calculation of $K$-groups and homotopy $K$-groups}\label{sec:calculation of K-theory} 
In this section, we compute various algebraic and homotopy $K$-groups. Let $X$ be a quasi-projective scheme over a field $k$. We say that a  closed subscheme $Y\hookrightarrow X$ is a conducting closed subscheme if the defining ideal sheaf of $Y$ maps to isomorphically to the ideal sheaf of $X^n$, where $X^n$ is the normalization of $X$.

\begin{lem}\label{quot-iso}
Let $k$ be an arbitrary field, and let $D$ be a reduced scheme over $k$. Then we can find a conductiong closed subscheme $E\hookrightarrow D$ such that the map
\begin{align*}
      f_{*}(\mathcal{K}_{2,D^n})/\mathcal{K}_{2,D}\to
   i_{*}(f'_{*}(\mathcal{K}_{2,E})/\mathcal{K}_{2,E'}),
\end{align*}
is the isomorphism of Nisnevich sheaves on $D$, where $D^n\to D $ is the normalization and $E'$ is defined by the following scheme cartesian square
 \begin{align*}
    \xymatrix{
      E' \ar@{^{(}->}[r]^{i'} \ar[d]_{f'} & D^{n} \ar[d]^{f}  \\
      E \ar@{^{(}->}[r]^{i} & D.}
    \end{align*}
\end{lem}

\begin{proof}
Take a conducting closed subscheme $E$. Then, there is a commutative diagram of exact sequence of sheaves on $D_{\Nis}$
\begin{align*}
   \xymatrix{
  \mathcal{K}_{2,(D,E)}\ar[d] \ar[r] & \mathcal{K}_{2,D} \ar[r] \ar[d] & 
  i_*(\mathcal{K}_{2,E}) \ar[r] \ar[d] & 0 \\
  f_{*}(\mathcal{K}_{2,(D^n,E')}) \ar[r] & f_{*}\mathcal{K}_{2,D^n} \ar[r] & 
  f_*i'_*(\mathcal{K}_{2,(E')}) \ar[r] & 0.}
\end{align*}
Here $\mathcal{K}_{i,D}$ is the sheaf associated to the presheaf  $U\to K_i(U)$ on $D_{\Nis}$, and others as well. Hence we get the following exact sequence of sheaves 
\begin{align*}
    \xymatrix{
  f_{*}(\mathcal{K}_{2,(D^n,E')})/ \mathcal{K}_{2,(D,E)}
  \ar[r] &
  f_{*}(\mathcal{K}_{2,D^n})/\mathcal{K}_{2,D} \ar[r] &
  i_{*}(f'_{*}(\mathcal{K}_{2,E'})/\mathcal{K}_{2,E}) \ar[r] & 0.
  }
\end{align*}
On the other hand, there is a homotopy fiber sequence of  double relative K-theory
\begin{align*}
   \xymatrix{
  \mathcal{K}_{2,(D,E)} \ar[r] & f_{*}(\mathcal{K}_{2,(D^n,E')}) \ar[r] & 
  \mathcal{K}_{1,(D,D^n,E)} \ar[r] & 
  \mathcal{K}_{1,(D,E)} \ar[r] & f_{*}(\mathcal{K}_{1,(D^n,E')}),
}
\end{align*}
where $K(D,D^n,E)$ is the homotopy fiber of the map $K(D,E)\to K(D^n,E')$. However, $\mathcal{K}_{1,(D,E)}$ is isomorphic to the sheaf $U\to Ker(\mathcal{O}^{\times}_U \to \mathcal{O}^{\times}_{U\times_{D}E})$ on $D_{\Nis}$, and $\mathcal{K}_{1,(D^n,E')}$ is similar. Hence $\mathcal{K}_{1,(D,E)} \to  f_{*}(\mathcal{K}_{1,(D^n,E')})$ is injective, or equivalently $\mathcal{K}_{2,(D^n,E')} \to \mathcal{K}_{1;(D,D^n,E)}$ is surjective. Therefore, as sheaves on $D_{\Nis}$ we have isomorphisms
\begin{align*}
    f_{*}(\mathcal{K}_{2,(D^n,E')})/ \mathcal{K}_{2,(D,E)} \cong
    \mathcal{K}_{1,(D,D^n,E)} \cong
    \mathcal{I}_E/\mathcal{I}^2_E \otimes_{E'} \Omega^1_{E'/E},
\end{align*}
where $\mathcal{I}_E$ is the ideal sheaf on $D$ defining $E$. The bijectivity of the right map was proved by Geller and Weibel \cite[Theorem 0.2]{GW}.

Since the same argument goes through for the closed subscheme $2E\hookrightarrow D$ which is determined by the ideal sheaf $\mathcal{I}_E^2$, we get a commutative diagram of sheaves 
\begin{align*}
    \xymatrix{
  \mathcal{I}_{2E}/\mathcal{I}^2_{2E} \otimes_{2E'} \Omega^1_{2E'/2E} \ar[r] \ar[d]^0^0 &
  f_{*}(\mathcal{K}_{2,D^n})/\mathcal{K}_{2,D} \ar[r] \ar@{=}[d] &
  i_{*}(f'_{*}(\mathcal{K}_{2,2E'})/\mathcal{K}_{2,2E}) \ar[r] \ar[d] & 0.\\
  \mathcal{I}_{E}/\mathcal{I}^2_{E} \otimes_{E'} \Omega^1_{E'/E} \ar[r] &
  f_{*}(\mathcal{K}_{2,D^n})/\mathcal{K}_{2,D} \ar[r] &
  i_{*}(f'_{*}(\mathcal{K}_{2,E'})/\mathcal{K}_{2,E}) \ar[r] & 0 }
\end{align*}
It is clear that the left vertical map is zero. This implies that  the surjective map $f_{*}(\mathcal{K}_{2,D^n})/\mathcal{K}_{2,D} \twoheadrightarrow  i_{*}(f'_{*}(\mathcal{K}_{2,2E'})/\mathcal{K}_{2,2E})$ is an isomorphism. Therefore $2E$ has our desired property.
\end{proof}

\begin{prop}\label{nor-seq}
Let $k$ be a perfect field of positive characteristic, and let $D$ be a reduced curve over $k$. Then we can find a conducting closed subscheme $E \hookrightarrow D$ such that there exists an exact sequence of abelian groups
\begin{align*}
    K_2(D^n)\oplus K_2(E)\to
    K_2(E'_{\red})\to SK_1(D)\to SK_1(D^n)\to 0,
\end{align*}
where $f:D^n\to D$ is the normalization of $D$ and $E':= E\times_{D} D^n$.
\end{prop}

\begin{proof}
  We can choose a conducting closed subscheme $E \hookrightarrow D$ such as Lemma \ref{quot-iso}. Then, there exists a commutative diagram of exact sequence of sheaves on $D_{\Nis}$
  \begin{align*}
      \xymatrix{
  \mathcal{K}_{2,D}\ar[d] \ar[r] &  f_{*}(\mathcal{K}_{2,D^n})\ar[r] \ar[d] & 
  f_{*}(\mathcal{K}_{2,D^n})/\mathcal{K}_{2,D} \ar[r] \ar[d]^{a} & 0 \\
  i_{*}\mathcal{K}_{2,E} \ar[r] & i_{*}(f'_{*}(\mathcal{K}_{2,E'})) \ar[r] & 
  i_{*}(f'_{*}(\mathcal{K}_{2,E'})/\mathcal{K}_{2,E}) \ar[r] & 0.}
  \end{align*}
  The right vertical map $a$ is an isomorphism by Lemma \ref{quot-iso}. Note that $\mathcal{K}_{2,D}\to f_{*}(\mathcal{K}_{2,D^n})$ is a generically isomorphism. Taking the cohomology of each sheaf on $D_{\Nis}$, we get a commutative diagram
 \begin{align*}
      \xymatrix{
  \, & H^0_{\Nis}(D,\mathcal{K}_{2,D^n})\ar[d] \ar[r] & H^0_{\Nis}(D,f_{*}(\mathcal{K}_{2,D^n})/\mathcal{K}_{2,D}) \ar[r] \ar[d]^{a}  & \, \\
  H^0_{\Nis}(E,\mathcal{K}_{2,E}) \ar[r] & 
  H^0_{\Nis}(E',\mathcal{K}_{2,E'}) \ar[r] & 
 H^0_{\Nis}(E,f'_{*}(\mathcal{K}_{2,D^n})/\mathcal{K}_{2,D})\ar[r] & \,
  } 
  \end{align*}
  \begin{align*}
\xymatrix{
     \ar[r] & H^1_{\Nis}(D,\mathcal{K}_{2,D}) \ar[r] \ar[d] & 
    H^1_{\Nis}(D^n,\mathcal{K}_{2,D^n}) \ar[r] &  0\\
     \ar[r] & 0}
  \end{align*}
Then the map $a$ is bijective as above.
The diagram chasing implies the following exact sequence
\begin{align*}
    H^0_{\Nis}(D,\mathcal{K}_{2,D^n})\oplus K_2(E)\to
    K_2(E')\to SK_1(D)\to SK_1(D^n)\to 0.
\end{align*}

This is because
\[ K_2(E)\to  H^0_{\Nis}(E,\mathcal{K}_{2,E}), \qquad
   K_2(E')\to  H^0_{\Nis}(E',\mathcal{K}_{2,E'}) \]
\[ H^1_{\Nis}(D,\mathcal{K}_{2,D}) \to SK_1(D), \qquad
   H^1_{\Nis}(D^n,\mathcal{K}_{2,D^n}) \to SK_1(D^n) \]
are all isomorphisms by the Thomason-Trobaugh spectral sequence.

Now, note that $\dim(E)=\dim(E')=0$ and $E'\cong \bigsqcup_i \Spec(A_i/I_i)$, where $A_i$ is a DVR and $I_i$ is an ideal of  $A_i$. Hence, by \cite[Lemma 3.4]{Denis-Stein},  $K_2(E')\cong K_2(E'_{\red})$. In addition, $K_2(D^n)\to H^0_{\Nis}(D,\mathcal{K}_{2,D^n})$ is surjective by the Thomason-Trobaugh spectral sequence. This completes the proof.
\end{proof}

From the previous lemma and proposition, we can clarify the connection between the algebraic $K$-theory and the homotopy $K$-theory.
\begin{cor}\label{SK-SKH iso}
Let $k$ be a perfect field of positive characteristic, and let $D$ be a reduced curve over $k$.
  Then the canonical map
\[
    SK_1(D) \to SKH_1(D)
\]
  is an isomorphism.
\end{cor}

\begin{proof}
 Since $KH$-theory satisfies $\cdh$-descent, we have an exact sequence of abelian groups 
 \begin{align*}
     \xymatrix{KH_2(D^n) \oplus KH_2(E_{\red}) \ar[r] & {KH_2(E'_{\red})} \ar[r] &  SKH_1(D) \ar[r] &
        SKH_1(D^n) \ar[r] & 0}
 \end{align*}
 induced by an abstract blow-up square. This follows from isomorphisms $KH_2(E')\cong KH_2(E'_{\red})$ and $KH_2(E)\cong KH_2(E_{\red})$. Combining this with the exact sequence of Proposition \ref{nor-seq}, we get the commutative diagram of exact sequences of abelian groups
 \begin{equation}\label{K-KH seq1}
  \xymatrix@C.8pc{
    K_2(D^n) \oplus K_2(E) \ar[r] \ar[d]^{\mbox{\fontsize{12pt}{22pt}\selectfont $a$}} & {K_2(E'_{\red})} \ar[r] \ar[d]^{\mbox{\fontsize{12pt}{22pt}\selectfont $b$}} &
    SK_1(D) \ar[r] \ar[d]^{\mbox{\fontsize{12pt}{22pt}\selectfont $c$}} &
    SK_1(D^n) \ar[r] \ar[d]^{\mbox{\fontsize{12pt}{22pt}\selectfont $d$}} &  0 \\
    KH_2(D^n) \oplus KH_2(E_{\red}) \ar[r] & {KH_2(E'_{\red})} \ar[r] &  SKH_1(D) \ar[r] &
    SKH_1(D^n) \ar[r] & 0.}
 \end{equation}

 The last right vertical map $d$ in \eqref{K-KH seq1} is an isomorphism from the following commutative diagram.
 \begin{align}\label{D^n:SK-SKH seq}
    \xymatrix{
    0 \ar[r] & SK_1(D^n) \ar[d]^{\mbox{\fontsize{12pt}{22pt}\selectfont $d$}} \ar[r] & K_1(D^n) \ar[r] \ar[d]^{\mbox{\fontsize{12pt}{22pt}\selectfont $e$}} & H^0_{\zar}(D^n, \mathcal{O}^{\times}_{D^n})\ar[r] \ar[d]^{\mbox{\fontsize{11pt}{22pt}\selectfont $f$}} &
    0 \\
    0 \ar[r] & SKH_1(D^n) \ar[r] & KH_1(D^n) \ar[r] &
    H^0_{\cdh}(D^n, \mathcal{O}^{\times}_{D^n}) \ar[r] & 
    0}
 \end{align}
 By Lemma \ref{Vorst} and Lemma \ref{seminormal}, we see that the middle map $e$ and right vertical map $f$ in \eqref{D^n:SK-SKH seq} are isomorphisms because $D^n$ is regular. Hence the left vertical map $d$ in \eqref{D^n:SK-SKH seq} and the last vertical map $d$ in \eqref{K-KH seq1} are also isomorphisms.
 
 The second vertical map $b$ in \eqref{K-KH seq1} is an isomorphism because $E'_{\red}$ is a finite product of spectra of fields, in particular a regular scheme. The first vertical map $a$ in \eqref{K-KH seq1} is surjective because $D^n$ is regular and the map $K_2(E)\to KH_2(E)$ factors through $K_2(E) \twoheadrightarrow K_2(E_{\red}) \stackrel{\cong}{\longrightarrow} KH_2(E)$. Hence a simple diagram chasing shows that the third vertical map $c$ in \eqref{K-KH seq1} is an isomorphism. This completes the proof.
\end{proof}

Finally, we conclude this section with the following proposition.
\begin{prop}\label{main prop}
Let $k$ be a perfect field of positive characteristic. Let $X$ be a projective normal surface over $k$, and let $D\hookrightarrow X$ be a reduced effective Cartier divisor on $X$. Assume that $X$ is $K_0$-regular. Then the canonical map 
\[
  K_0(X,D) \to KH_0(X,D)
\]
is injective.
\end{prop}

\begin{proof}
First, we have a commutative diagram of spectra
\begin{align*}
  \xymatrix{
   K(X.D)  \ar[r] \ar[d] & K(X) \ar[r] \ar[d] &  
  K(D) \ar[d] \\
  KH(X,D)  \ar[r] & KH(X) \ar[r] &  KH(D). }
\end{align*}
Taking their homotopy groups, we get a commutative diagram of exact sequences of abelian groups
\begin{align}\label{K-KH seq2}
    \xymatrix{
  K_1(X)  \ar[r] \ar[d]^{\mbox{\fontsize{12pt}{22pt}\selectfont $a$}} &
  {K_1(D)} \ar[r] \ar[d]^{\mbox{\fontsize{12pt}{22pt}\selectfont $b$}} &  
  K_0(X,D) \ar[r] \ar[d]^{\mbox{\fontsize{12pt}{22pt}\selectfont $c$}} &
  K_0(X) \ar[d]^{\mbox{\fontsize{12pt}{22pt}\selectfont $d$}}\\
  KH_1(X)  \ar[r] & {KH_1(D)} \ar[r] &  KH_0(X,D) \ar[r] &
  KH_0(X). }
\end{align}
The surjectivity of the first vertical map $a$ and the bijectivity of the last right vertical map $d$ in \eqref{K-KH seq2} follow from Lemma \ref{K-reg}. On the other hand, we have
\begin{equation}\label{D:SK-SKH seq}
\xymatrix{
  0 \ar[r] & 
  SK_1(D) \ar[d]^{\mbox{\fontsize{12pt}{22pt}\selectfont $e$}} \ar[r] & 
  K_1(D) \ar[r] \ar[d]^{\mbox{\fontsize{12pt}{22pt}\selectfont $b$}} &
  H^0_{\zar}(D, \mathcal{O}^{\times}_D)
      \ar[r] \ar[d]^{\mbox{\fontsize{11pt}{22pt}\selectfont $f$}} & 0 \\
    0 \ar[r] & SKH_1(D) \ar[r] & KH_1(D) \ar[r] & H^0_{\cdh}(D, \mathcal{O}^{\times}_D)
      \ar[r] & 0.}
\end{equation}
The left vertical map $e$ in \eqref{D:SK-SKH seq} is an isomorphism by Proposition \ref{SK-SKH iso}. The right vertical map $f$ in \eqref{D:SK-SKH seq} is injective from the isomorphism $H^0_{\cdh}(D, \mathcal{O}^{\times}_D)\cong H^0_{\zar}(D^{sn}, \mathcal{O}^{\times}_{D^{sn}})$ in Lemma \ref{seminormal}, where $D^{sn}$ is the seminormalization of $D$. From these, we see that the middle vertical map $b$ in \eqref{D:SK-SKH seq} and the second vertical map $b$ in \eqref{K-KH seq2} are injective. Therefore the diagram chasing shows the injectivity of the map $c$ in \eqref{K-KH seq2}.
\end{proof}

\section{Proof of the main theorem (1)}\label{sec:proof of 1}
In this section, we give a proof of Theorem \ref{main theorem} (1).

\begin{proof}[Proof of  Theorem \ref{main theorem} (1)]
By the resolution of singularities of surfaces, we can find a projective birational map $\pi:\overline{X}\to X$ such that $\overline{X}$ is regular and $(\pi^*D)_{\red} \hookrightarrow \overline{X}$ is simple normal crossing divisor, where $\pi^*D$ is a pullback of Cartier divisor $D$.

By Theorem \ref{Cisinski}, we have a homotopy cartesian square of spectra
\begin{align*}
    \xymatrix{
  KH(X) \ar[r] \ar[d] & {KH(D)} \ar[d]  \\
  KH(\overline{X}) \ar[r] & {KH(\pi^*D)}.}
\end{align*}
Hence the homotopy fibers of the top and bottom horizontal maps are homotopy equivalent; that is, 
\[ KH(X,D)\stackrel{\cong}{\longrightarrow} KH(\overline{X},\pi^*D). 
\]
We also have a homotopy equivalence $KH(\overline{X},\pi^*D){\cong} KH(\overline{X},(\pi^*D)_{\red})$.
Taking their homotopy groups, we get a commutative diagram
\begin{align*}
    \xymatrix{
  SK_0(X,D) \ar@{^{(}->}[r] \ar[d] &
  K_0(X,D) \ar[r]^{\mbox{\fontsize{12pt}{22pt}\selectfont $a$}} \ar[d] &
  KH_0(X,D)  \ar[d]^{\mbox{\fontsize{12pt}{22pt}\selectfont $c$}}\\
  SK_0(\overline{X},(\pi^*D)_{\red})  \ar@{^{(}->}[r]  & K_0(\overline{X},(\pi^*D)_{\red}) \ar[r]^{\mbox{\fontsize{12pt}{22pt}\selectfont $b$}}  &
  KH_0(\overline{X},(\pi^*D)_{\red}).}
\end{align*}
The right vertical map $c$ is an isomorphism from the above. Second horizontal maps $a,b$ in top and bottom rows are injective from Proposition \ref{main prop} because the pairs $(X,D)$ and $(\overline{X},(\pi^*D)_{\red}))$ satisfy the assumption of Proposition \ref{main prop}.

Therefore we get an injective map 
\begin{equation}\label{SK0 iso}
    SK_0(X,D) \hookrightarrow SK_0(\overline{X},(\pi^*D)_{\red}).
\end{equation}
By Bloch's formula for Chow groups, we have known the following isomorphisms
\[
\CH_0(X|D;X_{\sing})\stackrel{\cong}{\longrightarrow} 
H^2_{\tau}(X,\mathcal{K}^M_{2,(X,D)})\stackrel{\cong}{\longrightarrow}
H^2_{\tau}(X,\mathcal{K}_{2,(X,D)})\stackrel{\cong}{\longrightarrow}
SK_0(X,D),
\]
where $\tau$ denotes Zariski or Nisnevich topology. The first isomorphism is due to Gupta, Krishna, and Rathore \cite[Theorem 1.5]{GKR}. The second is due to Binda, Krishna, and Saito \cite[Lemma 3.2]{BKS}. Similarly, we have 
\[
H^S_0(X\setminus D) \stackrel{\cong}{\longleftarrow} 
\CH_0(\overline{X}|(\pi^*D)_{\red})\stackrel{\cong}{\longrightarrow} SK_0(\overline{X}|(\pi^*D)_{\red}),
\]
The first isomorphism is due to Binda and Krishna \cite[Theorem 1.1]{BK}. We shall consider a commutative diagram
\begin{align*}
    \xymatrix{
  \CH_0(X|D;X_{\sing})\ar@{->>}[r] \ar[d]_\cong & \CH_0(X|D) \ar@{->>}[r] & H^S_0(X \setminus D) \ar[d]^\cong\\
  SK_0(X,D) \ar@{^{(}->}[rr] && SK_0(\overline{X},(\pi^*D)_{\red}). }
\end{align*}
The vertical maps are isomorphism from the above. The horizontal map of the bottom row is injective by \eqref{SK0 iso}. Therefore the all maps in the above diagram are isomorphisms. This completes the proof.
\end{proof}

\section{Bertini theorems containing a subscheme}\label{sec:Bertini}
In this section, we recall the Bertini theorem containing a subscheme. Here we let $k$ be an infinite field. Let $X$ be a quasi-projective scheme over $k$. We choose an immersion $X\hookrightarrow \mathbb{P}^n_{k}$ and use the same immersion when no confusion occurs. We denote the scheme theoretic closure of $X$ in $\mathbb{P}^n_{k}$ by $\overline{X}$.

\begin{dfn}
Let $X, Z\hookrightarrow \mathbb{P}^n_{k}$ be subschemes such that $Z$ is closed in $\mathbb{P}^n_{k}$. We shall say that ``a general hypersurface $H\hookrightarrow \mathbb{P}^n_{k}$ containing $Z$ satisfies properties $\mathcal{P}$(e.g., smooth, $R_n$, $S_n$, and so on) for $X$" if, for any $d\gg 0$, there is a nonempty open subscheme $\mathcal{U}\subset |H^0(\mathbb{P}^n_{k},\mathcal{I}_Z(d))|$ such that $X\cap H$ satisfies $\mathcal{P}$ for any $H\in \mathcal{U}(k)$, where $\mathcal{I}_Z$ is the defining ideal sheaf of $Z\hookrightarrow \mathbb{P}^n_{k}$ in $\mathbb{P}^n_{k}$. If we can take $d=1$, we will use the terminology ``hyperplane" instead of ``hypersurface."
\end{dfn}
The following lemma is a well-known classical result.
\begin{lem}\label{Bertini for hyperplane}
    Let $X\hookrightarrow \mathbb{P}^n_k$ be a subscheme of $\mathbb{P}^n_k$. Then a general hyperplane $H\hookrightarrow \mathbb{P}^n_k$ satisfies that $X\cap H$ is regular along $X_{\reg}$.
\end{lem}
\begin{proof}
See \cite[Lemma 2.4]{GK-Bertini}.
\end{proof}

Since we cannot find appropriate references to the next lemma, we will prove it for the author's convenience. This result is mainly used in the next section.
\begin{lem}\label{Ra-Bertini}
Let $X\hookrightarrow \mathbb{P}^n_k$ be a subscheme, and let $Z\hookrightarrow X$ be a subscheme of $X$. Assume that $X$ is $R_a$-scheme and $\codim_X(Z)>a+1$. Then a general hypersurface $H\hookrightarrow \mathbb{P}^n_{k}$ containing $\overline{Z}$ satisfies that $X\cap H$ is $R_a$-scheme, where $\overline{Z}$ is the scheme theoretic closure of $Z$ in $\mathbb{P}^n_{k}$.
   
\end{lem}

\begin{proof}
Set $Y:=X\setminus \overline{Z}$. Let $d$ be a sufficiently large integer such that $\mathcal{I}_{\overline{Z}}(d)$ is generated by global sections. Then $\mathcal{I}_{\overline{Z}}(d)|_Y$ is an invertible $\mathcal{O}_Y$-module sheaf generated by global sections. Hence there is a map $Y\to \mathbb{P}(|H^0(\mathbb{P}^n_{k},\mathcal{I}_{\overline{Z}}(d))|)$. Combining this map with the Segre embedding, we get 
\begin{align*}   
    Y\to 
    \mathbb{P}(|H^0(\mathbb{P}^n_{k},\mathcal{I}_{\overline{Z}}(d))|)\times_k \mathbb{P}(W) \to
     \mathbb{P}(|H^0(\mathbb{P}^n_{k},\mathcal{I}_{\overline{Z}}(d)\otimes_k W)|),
\end{align*}
where $W:=|H^0(\mathbb{P}^n_{k},\mathcal{O}(1))|$. The composition of the maps is an immersion. This map is defined by the surjective map of $\mathcal{O}_Y$-module sheaf on $Y$
\begin{align*}
    |H^0(\mathbb{P}^n_{k},\mathcal{I}_{\overline{Z}}(d))|\otimes_k W \otimes_k \mathcal{O}_Y\to 
    \mathcal{I}_{\overline{Z}}(d+1)|_Y.
\end{align*}
Moreover this map factors through the map
\begin{align*}
    |H^0(\mathbb{P}^n_{k},\mathcal{I}_{\overline{Z}}(d+1))|\otimes_k \mathcal{O}_Y \to
    \mathcal{I}_{\overline{Z}}(d+1)|_Y.
\end{align*}
So, the latter map is also surjective. Let $U\subset \mathbb{P}(|H^0(\mathbb{P}^n_{k},\mathcal{I}_{\overline{Z}}(d+1))|)$ be the open subscheme such that we can define the map on $U$
\begin{align*}
    \mathbb{P}(|H^0(\mathbb{P}^n_{k},\mathcal{I}_{\overline{Z}}(d+1))|)\to 
    \mathbb{P}(|H^0(\mathbb{P}^n_{k},\mathcal{I}_{\overline{Z}}(d))|\otimes_k W).
\end{align*}
Then we have a commutative diagram 
\begin{align*}
    \xymatrix{
    Y \ar@{^{(}->}[r]^-{i} \ar[r] \ar[rd] \ar[d]_j &
      \mathbb{P}(|H^0(\mathbb{P}^n_{k},\mathcal{I}_{\overline{Z}}(d))|\otimes_k W)  \\
      \mathbb{P}(|H^0(\mathbb{P}^n_{k},\mathcal{I}_{\overline{Z}}(d+1))|) &  U \ar@{^{(}->}[l]  \ar[u]
     }
\end{align*}
Sine $i$ is an immersion, $Y\to U$ is also an immersion, and so is $j$. By Lemma \ref{Bertini for hyperplane}, $Y\cap H$ is regular along $Y_{\reg}$ for a general hyperplane $H\hookrightarrow \mathbb{P}(|H^0(\mathbb{P}^n_{k},\mathcal{I}_{\overline{Z}}(d+1))|)$. Moreover, we can choose such $H$ that $H$ does not contain any generic point of $Y_{\sing}$. Since $Y$ is an $R_a$-scheme, $\codim_Y(Y_{\sing})\geq a+1$, and therefore $\codim_{Y\cap H}(Y_{\sing}\cap H)\geq a+1$. Viewing a section $s\in |H^0(\mathbb{P}^n_{k},\mathcal{I}_{\overline{Z}}(d+1))|$ as a section of $|H^0(\mathbb{P}^n_{k},\mathcal{O}(d+1))|$, we see that $H$ defines a hypersurface containing $\overline{Z}$ with degree $d+1$ in $\mathbb{P}^n_k$. By the assumption, $\codim_{X\cap H}(Z)\geq a+1$. Therefore $X\cap H$ is an $R_a$-scheme for a general hypersurface $H\hookrightarrow \mathbb{P}^n_{k}$ containing $\overline{Z}$.
\end{proof}

By the next lemma, we can assume that $Z$ is a closed subscheme of $\overline{X}$.
\begin{lem}\label{GK-Bertini}
    Let $X, Z\hookrightarrow \mathbb{P}^n_{k}$ be subschemes such that $Z$ is closed in $\mathbb{P}^n_{k}$. Assume that  a general hypersurface $H\hookrightarrow \mathbb{P}^n_{k}$ containing $Z\cap \overline{X}$ satisfies properties $\mathcal{P}$ for $X$. Then  a general hypersurface $H\hookrightarrow \mathbb{P}^n_{k}$ containing $Z$ satisfies properties $\mathcal{P}$ for $X$.
\end{lem}
\begin{proof}
    See \cite[Lemma 2.5]{GK-Bertini}.
\end{proof}

Finally, we describe Bertini's theorem which were shown by Kleiman and Altman(see \cite{KA}).
\begin{dfn}
    Let $X$ be a scheme over $k$, and let $F$ be a coherent $\mathcal{O}_X$-module sheaf. For each integer $r\geq 0$, we put
    \begin{align*}
    \overline{X}(F,r) &:= \{x\in X \mid \dim_{k(x)}(F\otimes_{\mathcal{O}_X}   k(x))\geq r\}\\
    X(F,r) &:= \overline{X}(F,r)\setminus \overline{X}(F,r+1).
    \end{align*}
    
\end{dfn}

\begin{thm}\label{AK-Bertini}
    Let $k$ be an infinite field. Let $X$ be a projective scheme over $k$, and let $Z\hookrightarrow X$ be a closed subscheme of $X$. Assume that $X_{\reg}$ is smooth over $k$ and the following inequality is satisfied
    \begin{align*}
\underset{e\geq 0}{\max}(\dim(Z\cap X_{\reg}(\Omega^1_{Z\cap X_{\reg}},e))+e)< \dim(X)
   \end{align*}
   Then a general hypersurface $H\hookrightarrow \mathbb{P}^n_{k}$ containing $Z$ satisfies that $H\cap X_{\reg}$ is a smooth scheme over $k$.
\end{thm}
\begin{proof}
    See \cite[Theorem 7]{KA}.
\end{proof}

\section{Proof of the main theorem (2)}\label{sec:proof of 2}

\subsection{Proof of the main theorem (2): for infinite fields}

  First, we will prove our claim for the case that $k$ is an infinite field.
\begin{proof}
  Let $p:C \to X$ be a finite map and let $f\in \Ker(\mathcal{O}^{\times}_{C,p^{*}D}\to \mathcal{O}^{\times}_{(p^{*}D)_{\red}})$, which give rational equivalence of Suslin homology $H^S_0(X\setminus D)$. That is, $C$ is an integral normal curve, the image of $p$ is not contained in $D$, and $p_*(\divisor(f))=0\in H^S_0(X\setminus D)$. To prove the injectivity of the canonical map $\CH_0(X|D)\to H^S_0(X\setminus D)$, it suffices to show that $p_*(\divisor(f))=0\in \CH_0(X|D)$.

We shall consider the following diagram:
\begin{align*}
  \xymatrix{
C \ar@{^{(}->}[r] \ar[dr]_-{p} & C\times_{k}X \ar[d]^-{\pi} \\
& X.}  
\end{align*}
 Since $C$ is smooth and there is a pushforward map
 \[
 \pi_*: \CH_0(C\times_{k}X|C\times_{k}D)\to \CH_0(X|D),
 \]
 we may assume that $p: C\to X$ is a closed immersion.

If $\dim(X)\leq 2$, our claim has been proved in the Section \ref{sec:proof of 1}. So we may assume that $\dim(X)> 2$, and hence $\dim(D)> 1$. Set a closed immersion $X\hookrightarrow \mathbb{P}^n_k$. Since $k$ is infinite, by the Bertini theorem containing $C$ for $X\hookrightarrow \mathbb{P}^n_k$ (see Theorem \ref{AK-Bertini}), a general hypersurface $H\hookrightarrow \mathbb{P}^n_k$ containing $D$ satisfies that $X\cap H$ is a smooth projective scheme. Moreover, by the Bertini theorem containing $C\cap D$ for $D\hookrightarrow \mathbb{P}^n_k $ (see Lemma \ref{Ra-Bertini}), a general hypersurface $H\hookrightarrow \mathbb{P}^n_k$ containing $D$ satisfies that $D\cap H$ is an $R_0$-scheme. Note that $D$ is an $R_0$-scheme by the hypothesis that $D$ is reduced. Furthermore, since $\dim(D\cap C)=0$ and  $\codim_D(D\cap C)> 1 $, we can use Lemma \ref{Ra-Bertini}. Now $X$ is smooth, in particular Cohen-Macaulay, and $D\hookrightarrow X$ is an effective Cartier divisor. Hence $D$ is also Cohen-Macaulay, and so is $D\cap H$ for a general hypersurface $H\hookrightarrow \mathbb{P}^n_k$. Thus a general hypersurface $H\hookrightarrow \mathbb{P}^n_k$ containing $D\cap C$ satisfies that $D\cap H$ is reduced. Using Lemma \ref{GK-Bertini}, we conclude that a general hypersurface $H\hookrightarrow \mathbb{P}^n_k$ containing $C$ satisfies that $X\cap D$ is smooth and $D\cap H$ is reduced. Iterating the same argument, finally we get a (regular) close immersion $\pi:X'\hookrightarrow X$ such that 
\begin{enumerate}
    \item $X'$ is projective smooth surface over $k$.
    \item $D':=X'\cap D\hookrightarrow X'$ is a reduced  effective Cartier divisor on $X'$.
    \item $X'$ contains $C$.
\end{enumerate}
\begin{align*}
    \xymatrix{
    C \ar@{^{(}->}[r]^{p'} \ar@(ur,ul)@{^{(}->}[rr]^p &
    X' \ar[r]^{\pi} & X
    }
\end{align*}
Since $p'_*(\divisor(f))=0\in H_0(X'\setminus D')$ holds by the definition, we get $p'_*(\divisor(f))=0\in \CH_0(X'|D')$ by Theorem \ref{main theorem} for a surface case. Hence we have
\[ 
p_*(\divisor(f))=\pi_*(p'_*(\divisor(f)))=0\in \CH_0(X|D). 
\]
This completes the proof of the main theorem when $k$ is an infinite field.
\end{proof}

\subsection{Proof of the main theorem (2): for finite fields}

Next, we assume that $k$ is a finite field. We put $X_F:=X\times_k F$ and $D_F:=X\times_k D$ for a field extension $F/k$. 
\begin{proof}
    Choose two primes $\ell_1\neq \ell_2$ different from char$(k)$. Let $k_i/k$ be the pro-$\ell_i$ field extension, that is $k_i=\underset{|F:k|=\ell_i^{s}}{\bigcup}F$.
Then we get canonical isomorphisms
\[
    \CH_0(X_{k_i}|D_{k_i})\cong 
    \underset{|F:k|=\ell_i^{s}}{\varinjlim}\CH_0(X_F|D_F), \quad
     H^S_0(X_{k_i}\setminus D_{k_i})\cong 
    \underset{|F:k|=\ell_i^{s}}{\varinjlim}H^S_0(X_F\setminus D_F),
\]
and a commutative diagram
\begin{align*}
    \xymatrix{
    \CH_0(X|D) \ar[r] \ar[d] & H^S_0(X\setminus D) \ar[d]\\
     \CH_0(X_{k_i}|D_{k_i}) \ar[r]  & H^S_0(X_{k_i}\setminus D_{k_i}). 
    }
\end{align*}

Let $\alpha \in \CH_0(X|D)$ be an element such that $\alpha=0\in H^S_0(X\setminus D)$. The above isomorphism implies that, for $i=1,2$, there exists a finite field extension $F_i/k$ whose degree is a power of $\ell_i$ such that $\alpha=0\in \CH_0(X_{F_i}|D_{F_i})$. 

We shall consider the following commutative diagram:
\begin{align*}
    \xymatrix{
    \CH_0(X|D) \ar[r] \ar[d] & H^S_0(X\setminus D) \ar[d]\\
     \CH_0(X_{F_i}|D_{F_i}) \ar[r] \ar[d] & H^S_0(X_{F_i}\setminus D_{F_i}) \ar[d]\\
 \CH_0(X|D) \ar[r]& H^S_0(X\setminus D)   
    }
\end{align*}

By the projection formula, the composition of vertical maps of the left and the right are multiplication by $\ell_i ^{n_i}$, where  $\ell_i ^{n_i}=|F_i:k|$ for $i=1,2$ (see \cite[Theorem 8.5]{Gupta-Krishna} and \cite[after Examples 2.7]{Schmidt}). Thus we get $\ell ^{n_i} \alpha=0\in \CH_0(X|D)$ for $i=1,2$. Since $\ell_1$ and  $\ell_2$ are different primes, we see that $\alpha=0\in \CH_0(X|D)$. This completes the proof of the finite field case .
\end{proof}

\end{document}